\documentclass{article}
\usepackage[unicode=true,
 bookmarks=true,bookmarksnumbered=false,bookmarksopen=false,
 breaklinks=false,pdfborder={0 0 0},backref=false,colorlinks=true]
 {hyperref}
\hypersetup{
 pdfauthor={Yaozhong Hu-Khoa Le}}

\usepackage{amsfonts}
\usepackage{color}
\usepackage{amsmath}

\newtheorem{theorem}{Theorem}[section]

\newtheorem{defn}[theorem]{Definition}
\newtheorem{thm}[theorem]{Theorem}
\newtheorem{prop}[theorem]{Proposition}
\newtheorem{cor}[theorem]{Corollary}

\newtheorem{example}[theorem]{Example}

\newtheorem{lem}[theorem]{Lemma}

\newtheorem{problem}[theorem]{Problem}

\newtheorem{rem}[theorem]{Remark}

\newenvironment{proof}[1][Proof]{\noindent\textbf{#1.} }{\ \rule{0.5em}{0.5em}}

\def\RR{\mathbb{R}}

\def\EE{\mathbb{E}}

\def\cF{{\cal F}}

\def\be{{\beta}}
\def\de{{\delta}}

\def\al{{\alpha}}

\def\be{{\beta}}

\def\de{{\delta}}

\def \eref#1{\hbox{(\ref{#1})}}

\global\long\def\ee#1{e^{-\frac{\delta^2}{2#1}}}

\begin{document}

\title{A multiparameter Garsia-Rodemich-Rumsey inequality and some
applications}

\author{     Yaozhong {\sc Hu}\thanks{Y.  Hu is
partially supported by a grant from the Simons Foundation
\#209206.
 }
\ \    and \  \       Khoa   {\sc L\^e}  \\
Department of Mathematics\thinspace ,\ University of Kansas\\
405 Snow Hall\thinspace ,\ Lawrence, Kansas 66045-2142\\
}
\date{}
\maketitle

\begin{abstract}
We extend the classical Garsia-Rodemich-Rumsey inequality
to the multiparameter situation. The new inequality is applied
to obtain some joint H\"older continuity along the rectangles
for  fractional Brownian
fields $W(t, x)$ and for  the solution $u(t, y)$ of stochastic heat equation with additive white noise.
\end{abstract}

\section{Introduction}

Let  the function $\Psi: [0, \infty)\rightarrow [0, \infty)$ be
non decreasing  with  $\displaystyle \lim_{u\to\infty}\Psi(u)=\infty$  and let
the function $p
:[0, 1]\rightarrow [0, 1]$  be
continuous and non decreasing     with  $p(0)=0$.    Set
\[
\left\{
\begin{array}{ll}
\Psi^{-1}(u)=\sup_{\Psi(v)\le u}v &\qquad     \hbox{if \ $ \Psi(0)\le u<\infty $}\\ \\
p^{-1}(u)=\max_{p(v)\le u}v  &\qquad   \hbox{if \  $0\le u\le p(1)$}
\end{array}\right.
\]
The celebrated   Garsia-Rodemich-Rumsey inequality
\cite{garsiarodemich} takes the following form:
\begin{lem}
\label{lem:grr}  Let $f$ be a continuous function on $[0,1]$ and
suppose that
\[
\int_{0}^{1}\int_{0}^{1}\Psi\left(\frac{|f(x)-f(y)|}{p(x-y)}\right)dxdy\le B<\infty\,.
\]
Then for all $s,t\in[0,1]$ we have
\begin{equation}
|f(s)-f(t)|\le8\int_{0}^{|s-t|}\Psi^{-1}\left(\frac{4B}{u^{2}}\right)dp(u).\label{ineq:GRR}
\end{equation}
\end{lem}

This Garsia-Rodemich-Rumsey  lemma  \ref{lem:grr}  is very powerful
in the study of the sample path  H\"older continuity of a stochastic process and
in   other occasions.
For example if
 $\Psi(u)=|u|^{p}$ and $p(u)=|u|^{\alpha+1/p}$,  where $p\alpha>1$,
the inequality (\ref{ineq:GRR}) implies  the following Sobolev imbedding
inequality
\begin{equation}
|f(s)-f(t)|\le C_{\alpha,p}|t-s|^{\alpha-1/p}\left(\int_{0}^{1}\int_{0}^{1}
\frac{|f(x)-f(y)|^{p}}{|x-y|^{\alpha p+1}}dxdy\right)^{1/p}.\label{ineq:sobolev}
\end{equation}
The Garsia-Rodemich-Rumsey lemma has been extended to several
parameter or infinite many parameters. However the parameter space are
assumed to have a  distance (metric space)
and the Garsia-Rodemich-Rumsey lemma is
with respect to that  distance.  This method immediately yields the following result for a fractional Brownian field
$W^{H}( x)$ of Hurst parameter $H=( H_1, \cdots, H_d)$, then for any $\be_i$ with $\be_i<H_i$, $i=  1, \cdots,
d$, one has
\begin{equation}\label{e.wxy}
|W(y)-W( x)|\le L  \sum_{i=1}^d |y_i-x_i|^{\be_i} \,,
\end{equation}
where $L$ is an integrable  random variable.  One  can  improve this
result (Remark \ref{rem:1}) by  our version of multiparameter Garsia-Rodemich-Rumsey inequality. We do not seek for a suitable metric but rather deal directly with the multidimensional nature of the parameter space.

Let us explain
our motivation by considering the two parameter fractional Brownian
field $\{ W(x_1, x_2), (x_1, x_2)\in [0,1]^2\}$ of Hurst parameter
$H=(H_1, H_2)$.  Given two points $x$ and $y$ in $\RR^2$, we consider
the increment of $W$ along with the rectangle determined by $x=(x_1, x_2)$ and
$y=(y_1, y_2)$:
\begin{equation}
\square W:=W(y_1, y_2) -W(x_1, y_2)-W(x_2, y_1)+W(x_1, x_2)\,.
\label{e.1.3}
\end{equation}
In \cite{Ralchenko}, using a two-parameter version \eqref{ineq:sobolev}, the author showed that for any  $\be_1, \be_2$ with $\be_1<H_1$  and
$\be_2<H_2$, there is an integrable random  constant $L_{\be_1,
\be_2}$ such that
\begin{equation}
|\square W|\le L_{\be_1,
\be_2}|y_1-x_1|^{\be_1}|y_2-x_2|^{\be_2}\,.\label{e.1.4}
\end{equation}
The above result was also obtained in \cite{ayache-antoine-leger} based on a two-parameter version of Kolmogorov continuity theorem. Along the paper (in Corollary \ref{cor:fbm}), we shall see that the following sharper inequality than \eqref{e.1.4} holds
\begin{equation}
|\square W|\le L_{H_1, H_2} |y_1-x_1|^{ H_1}|y_2-x_2|^{H_2}
\sqrt{\left| \log \left(|y_1-x_1| |y_2-x_2| \right)\right| }\,.\label{e.1.5}
\end{equation}
Consequently, this estimate implies
\begin{multline*}
|W(x_1,x_2)-W(y_1,y_2)|\le L_{H_1, H_2}\left(|x_1-y_1|^{H_1}|x_2|^{H_2}\sqrt{\left|\log (|x_1-y_1||x_2|)\right|}\right.\\
\left.+|x_1|^{H_1}|x_2-y_2|^{H_2}\sqrt{\left|\log (|x_2-y_2||x_1|)\right|} \right)
\end{multline*}
which improves \eqref{e.wxy}. 
We shall call such property as in \eqref{e.1.5} or \eqref{e.1.4} joint H\"older continuity. It turns out that a large class of Gaussian fields enjoys sample path joint H\"older continuity (Theorem \ref{thm:contW}.)

Our method is first formulate and
prove  a multiparameter version of the classical Garsia-Rodemich-Rumsey inequality (\ref{ineq:GRR}). The generalized inequality is then applied to obtain sample path joint H\"older continuity for random fields. Our result   generalizes the results in
\cite{garsiarodemich}, \cite{Ralchenko} and provides a different approach  for sample path property problem of random fields (compare to the approach in \cite{ayache-antoine-leger}, \cite{ayache-xiao} and \cite{xiao}.)

The paper is structured as follows. In Section \ref{sec:grr}, we shall state and prove our multiparameter version of
the Garsia-Rodemich-Rumsey lemma.  The idea is to use induction on the dimension of the parameter space
after some observations of the property of operator $\square$
defined by \eref{e.1.3}.  Some
part of the proof is similar to the original proof of
Garsia-Rodemich-Rumsey \cite{garsiarodemich} with some modification.
However, we feel it is more appropriate to give a  detailed proof.

In Section \ref{sec:random-fields}, we introduce a multiparameter
version of Kolmogorov continuity criteria (Theorem \ref{thm:Kolmogorov}).
To our best knowledge, a two-parameter of Theorem \ref{thm:Kolmogorov}
first appeared in \cite{ayache-antoine-leger}.

Section \ref{sec:gaussian-fields} is devoted for the study of sample path joint continuity for
Gaussian fields. We give a sufficient condition for a Gaussian field to
 possess sample path joint continuity (Theorem \ref{thm:contW}).   We also derive
 the estimate \eref{e.1.5} for fractional Gaussian field.  In Section
 \ref{sec:spde}, we shall study
the joint H\"older continuity of solution of a stochastic  heat equation with
additive space-time white noise.

\setcounter{equation}{0}
\section{Multiparameter Garsia-Rodemich-Rumsey\\ inequality}\label{sec:grr}

We state the following technical lemma which generalizes  a crucial
argument used in \cite{garsiarodemich} in  the proof of Lemma
\ref{lem:grr}.
\begin{lem}
\label{lem:tk}Let $(\Omega\,, \cF) $ be a measurable space and let
$\mu$ be a positive measure on $(\Omega\,, \cF)$.  Let
$g:\Omega\times[0,1]\to\mathbb{R}^{m}$ be a measurable function such
that
\[
\int_{0}^{1}\int_{0}^{1}\int_{\Omega}\Psi\left(\frac{|g(z,t)-g(z,s)|}{p(|t-s|)}\right)\mu(dz)dsdt\le B<\infty.
\]
Then there exist two decreasing sequences $\{t_{k}\,, k=0, 1,   \cdots\}$ and $\{d_{k}\,, k=0, 1,   \cdots\}$ with
\begin{equation}
t_{k}\le d_{k-1}=p^{-1}\left(\frac{1}{2}p(t_{k-1})\right)\,, \quad k=1, 2, \cdots\label{eq:tdk}
\end{equation}
such that  the following inequality holds
\begin{equation}
\int_{\Omega}\Psi\left(\frac{|g(z,t_{k})-g(z,t_{k-1})|}{p(|t_k-t_{k-1}|)}\right)
\mu(dz)\le\frac{4B}{d_{k-1}^{2}}\,.  \label{ineq:gtn}
\end{equation}
\end{lem}
\begin{proof}
We follow the argument in \cite{garsiarodemich}. Let
\[
I(t)=\int_{0}^{1}\int_{\Omega}\Psi\left(\frac{|g(z,t)-g(z,s)|}{p(|t-s|)}\right)\mu(dz)ds.
\]
From  the  assumption $\int_{0}^{1}I(t)dt\le B$    it follows that
there is some  $t_{0}\in(0,1)$ such that
\[
I(t_{0})\le B.
\]
Now we can describe how to obtain the sequences $d_k$ and $t_k$
recursively for $k=1, 2, \cdots$. Given $t_{k-1}$,  define
\[
d_{k-1}=p^{-1}\left(\frac{1}{2}p(t_{k-1})\right)\,.
\]
Then we  choose $t_{k}\le d_{k-1}$   such that
\begin{equation}
I(t_{k})\le\frac{2B}{d_{k-1}}\label{ineq:itk}
\end{equation}
and
\begin{equation}
\int_{\Omega}\Psi\left(\frac{|g(z,t_{k})-g(z,t_{k-1})}{p(|t_{k}-t_{k-1}|)}\right)
\mu(dz)\le\frac{2I(t_{k-1})}{d_{k-1}}\,.\label{ineq:psitk}
\end{equation}
It is always possible to find $t_k$ such that the inequalities
\eref{ineq:itk} and \eref{ineq:psitk} hold simultaneously, since
each of the two inequalities can be violated only on a set of
$t_{k}$'s of measure strictly less than $\frac{1}{2}d_{k-1}$.   Now
(\ref{ineq:itk}) and (\ref{ineq:psitk}) gives
\[
\int_{\Omega}\Psi\left(\frac{|g(z,t_{k})-g(z,t_{k-1})|}{p(|t_k-t_{k-1}|)}\right)\mu(dz)\le
\frac{2I(t_{k-1})}{d_{k-1}}\le \frac{4B}{ d_{k-1}d_{k-2}} \le
 \frac{4B}{  d_{k-1}^2}\,.
\]
This is \eref{ineq:gtn}.
\end{proof}

Let $x=(x_{1},\dots,x_{n})$ and $y=(y_{1},\dots,y_{n})$ be in
$\mathbb{R}^{n}$. We denote $x'=(x_{1},\dots,x_{n-1})$ and
$y'=(y_{1},\dots,y_{n-1})$.  For each integer $k=1, 2, \cdots, n$,
we define
\[
V_{k,y}x=(x_{1},\dots,x_{k-1},y_{k},x_{k+1},\dots,x_{n}).
\]
Let $f$ be a function from $\mathbb{R}^{n}$ to $\mathbb{R}^{m}$. We
define  the operator $V_{k,y}$ acting  on $f$ in the  following way:
\[
V_{k,y}f(x)=f(V_{k,y}x)
\]
or $V_{k,y}f=f\circ V_{k,y}$ in short. It is straight forward to
verify  that
\[
V_{k,y}V_{k,y}=V_{k,y}
\]
and
\[
V_{k,y}V_{l,y}=V_{l,y}V_{k,y}
\]
for $k\neq l$. Next,  we  define the joint increment of a  function $f$ on an $n$-dimensional rectangle, namely
\[
\square_{y}^{n}f(x)=\prod_{k=1}^{n}(I-V_{k,y})f(x)
\]
where $I$ denotes the identity operator.
\begin{example} If $n=2$,  then it is easy to see
that $\square_{y}^{2}f(x)=f(y_1,y_2)-f(x_1,y_2)-f(y_1, x_2)+f(x_1,
y_2)$, which is the increment of $f$ over the rectangle  containing
the two points $x$ and $y$ with all sides parallel to the axis. In
particular, if $f(x_1, x_2)=x_1x_2$,  then
$\square_{y}^{2}f(x)=(x_1-y_1)(x_2-y_2)$, which is the area of the
rectangle. In a more general case, when $f$ has the form $f(x)=\prod_{j=1}^nf_j(x_j)$, then
\[\square_y^nf(x)=\prod_{j=1}^n[f_j(x_j)-f_j(y_j)]\,.\]
\end{example}
The following simple identity enable us to show our theorem by
induction and plays an essential role in our approach:
\begin{eqnarray}
\square_{y}^{n}f(x) & = & \prod_{k=1}^{n-1}(I-V_{k,y})f(x)-V_{n,y}
\prod_{k=1}^{n-1}(I-V_{k,y})f(x)\nonumber \\
 & = &
 \square_{y'}^{n-1}f(x',x_{n})-\square_{y'}^{n-1}f(x',y_{n}).\label{e.2.5}
\end{eqnarray}

We are now in the position to state our general version of Lemma \ref{lem:grr}.
\begin{thm}
\label{lem:grrn}Let $f(x)$ be a continuous function on $[0,1]^{n}$
and suppose that
\begin{equation}
\int_{[0,1]^{n}}\int_{[0,1]^{n}}\Psi\left(\frac{|\square_{y}^{n}f(x)|}
{\prod_{k=1}^{n}p_{k}(|x_{k}-y_{k}|)}\right)dxdy\le B<\infty\,.
\label{cond.b}
\end{equation}
Then for all $s,t\in[0,1]^{n}$  we have
\begin{equation}
|\square_{s}^{n}f(t)|\le8^{n}\int_{0}^{|s_{1}-t_{1}|}\cdots\int_{0}^{|s_{n}-t_{n}|}
\Psi^{-1}\left(\frac{4^{n}B}{u_{1}^{2}\cdots u_{n}^{2}}\right)dp_{1}(u_{1})\cdots dp_{n}(u_{n}).\label{ineq:grr1}
\end{equation}
\end{thm}
\begin{proof}
We proceed by induction on $n$. For $n=1$, it coincides with the
original Garsia-Rodemich-Rumsey inequality (\ref{ineq:GRR}). Suppose
(\ref{ineq:grr1}) holds for $n-1$. Let $f$ be a continuous function
on $[0,1]^{n}$. For any   $x',y'\in\mathbb{R}^{n-1}$ and any
$s\in[0,1]$,  put
\[
g(x',y',s)=\frac{\square_{y'}^{n-1}f(x',s)}{\prod_{k=1}^{n-1}p_{k}(|x_{k}'-y_{k}'|)}.
\]
Let $\Omega=[0,1]^{n-1}\times[0,1]^{n-1}$, $z=(x',y')$.  By
\eref{e.2.5}  we can rewrite (\ref{cond.b}) as
\[
\int_{0}^{1}\int_{0}^{1}\int_{\Omega}\Psi\left(\frac{|g(z,s)-g(z,t)|}{p_{n}(|s-t|)}\right)dzdsdt\le B<\infty.
\]
Applying  Lemma \ref{lem:tk}, we can find sequences $\{t_{k}\}$ and
$\{d_{k}\}$ such that
\begin{equation}
t_{k}\le d_{k-1}=p_{n}^{-1}\left(\frac{1}{2}p_{n}(t_{k-1})\right)\label{eq:tdk-1}
\end{equation}
and
\begin{equation}
\int_{\Omega}\Psi\left(\frac{|g(z,t_{k})-g(z,t_{k-1})|}{p_{n}(t_{k}-t_{k-1})}\right)dz
\le\frac{4B}{d_{k-1}^{2}}.\label{ineq:psigtk}
\end{equation}
For each $k\in\mathbb{N}$ and $x'\in[0,1]^{n-1}$, let
\[
h_{k}(x')=\frac{f(x',t_{k})-f(x',t_{k-1})}{p_{n}(|t_{k}-t_{k-1}|)}.
\]
Again from \eref{e.2.5} it follows
\begin{eqnarray*}
\frac{\square_{y'}^{n-1}h_{k}(x')}{\prod_{i=1}^{n-1}p_{i}(|x_{i}'-y_{i}'|)} & = &
\frac{\square_{y'}^{n-1}f(x',t_{k})-\square_{y'}^{n-1}f(x',t_{k-1})}
{\prod_{i=1}^{n-1}p_{i}(|x_{i}'-y_{i}'|)p_{n}(|t_{k}-t_{k-1}|)}\\
 & = & \frac{g(x',y',t_{k})-g(x',y',t_{k-1})}{p_{n}(|t_{k}-t_{k-1}|)}\,.
\end{eqnarray*}
Thus,  the inequality (\ref{ineq:psigtk}) becomes
\[
\int_{[0,1]^{n-1}}\int_{[0,1]^{n-1}}\Psi\left(\frac{\square_{y'}^{n-1}h_{k}(x')}
{\prod_{i=1}^{n-1}p_{i}(|x_{i}'-y_{i}'|)}\right)dx'dy'\le\frac{4B}{d_{k-1}^{2}}.
\]
Now,  by our induction hypothesis, for every $k\ge1$,
$a,b\in[0,1]^{n-1}$, $a=(a_{1},\dots,a_{n-1})$ and
$b=(b_{1},\dots,b_{n-1})$,
\begin{eqnarray*}
|\square_{a}^{n-1}h_{k}(b)|
&\le&8^{n-1}\int_{0}^{|a_{1}-b_{1}|}\cdots\int_{0}^{|a_{n-1}-b_{n-1}|}\Psi^{-1}\left(\frac{4^{n}B}{u_{1}^{2}\cdots
u_{n-1}^{2}d_{k-1}^{2}}\right)\\
&&\qquad dp_{1}(u_{1})\cdots dp_{n-1}(u_{n-1})\,.
\end{eqnarray*}
Denoting
$A=[0,|a_{1}-b_{1}|]\times\cdots\times[0,|a_{n-1}-b_{n-1}|]$ and
$dp(u_{1},\cdots,u_{n-1})=dp_{1}(u_{1})\cdots dp_{n-1}(u_{n-1})$,
the above inequality can be rewritten as
\begin{equation}
|\square_{a}^{n-1}f(b,t_{k})-\square_{a}^{n-1}f(b,t_{k-1})|
\le8^{n-1}\int_{A}\Psi^{-1}\left(\frac{4^{n}B}{u_{1}^{2}\cdots u_{n-1}^{2}d_{k-1}^{2}}\right)
dp(u_{1},\cdots,u_{n-1})p_{n}(t_{k-1}-t_{k}).\label{ineq:fab}
\end{equation}
On the other hand, by (\ref{eq:tdk-1}),  we have
\begin{eqnarray*}
p_{n}(t_{k-1}-t_{k}) & \le & p_{n}(t_{k-1})\\
 & = & 2p(d_{k-1})\le4\left[p(d_{k-1})-p(d_{k})\right].
\end{eqnarray*}
Combining this inequality with (\ref{ineq:fab}) yields
\begin{eqnarray*}
&&|\square_{a}^{n-1}f(b,t_{0})-\square_{a}^{n-1}f(b,0)|\\
&\le& \sum_{k=1}^{\infty}|\square_{a}^{n-1}f(b,t_{k})-\square_{a}^{n-1}f(b,t_{k-1})|\\
&\le& 8^{n-1}\sum_{k=1}^{\infty}4\left[p(d_{k-1})-p(d_{k})\right]\int_{A}\Psi^{-1}\left(\frac{4^{n}B}{u_{1}^{2}\cdots u_{n-1}^{2}d_{k-1}^{2}}\right)dp(u_{1},\cdots,u_{n-1})\\
&\le& 8^{n-1}\sum_{k=1}^{\infty}4\int_{d_{k}}^{d_{k-1}}\int_{A}\Psi^{-1}\left(\frac{4^{n}B}{u_{1}^{2}\cdots u_{n-1}^{2}u_{n}^{2}}\right)dp(u_{1},\cdots,u_{n-1})dp_{n}(u_{n})\\
&\le&
8^{n-1}4\int_{0}^{1}\int_{A}\Psi^{-1}\left(\frac{4^{n}B}{u_{1}^{2}\cdots
u_{n}^{2}}\right)dp(u_{1},\cdots,u_{n-1})dp_{n}(u_{n}).
\end{eqnarray*}
With $f(x',1-x_{n})$ replaced  $f(x',x_{n})$ we can obtain the same
bound for
\[
|\square_{a}^{n-1}f(b,t_{0})-\square_{a}^{n-1}f(b,1)|.
\]
Hence,  for every $a,b\in[0,1]^{n-1}$,
\begin{equation}
|\square_{a}^{n-1}f(b,1)-\square_{a}^{n-1}f(b,0)|\le8^{n}\int_{0}^{1}\int_{A}\Psi^{-1}\left(\frac{4^{n}B}{u_{1}^{2}\cdots u_{n}^{2}}\right)dp(u_{1},\cdots,u_{n-1})dp_{n}(u_{n}).\label{ineq:fab01}
\end{equation}
To obtain (\ref{ineq:grr1}) for general $s,t$ in $[0,1]^{n}$, we
set
\[
\bar{f}(t',\tau)=f(t',s_{n}+\tau(t_{n}-s_{n}))\hbox{ for $\tau\in[0,1]$}
\]
and
\[
\bar{p}_{n}(u)=p_{n}(u|s_{n}-t_{n}|).
\]
Upon restricting the range of the integration in (\ref{cond.b}) and
carrying out a change of variables we get
\[
\int_{[0,1]^{n}}\int_{[0,1]^{n}}\Psi\left(\frac{|\square_{y}^{n}\bar{f}(x)|}{\prod_{k=1}^{n-1}p_{k}(|x_{k}-y_{k}|)\bar{p}_{n}(|x_{n}-y_{n}|)}\right)dxdy\le\frac{B}{|s_{n}-t_{n}|^{2}}.
\]
Thus, by (\ref{ineq:fab01}), we deduce
\begin{eqnarray*}
|\square_{s}^{n}f(t)|
&=&|\square_{s'}^{n-1}\bar{f}(t',1)-\square_{s'}^{n-1}\bar{f}(t',0)|\\
&\le&
8^{n}\int_{0}^{1}\int_{0}^{|s_{1}-t_{1}|}\cdots\int_{0}^{|s_{n-1}-t_{n-1}|}
\Psi^{-1}\left(\frac{4^{n}B}{u_{1}^{2}\cdots
u_{n}^{2}|s_{n}-t_{n}|^{2}}\right)\\
&&\qquad dp(u_{1},\cdots,u_{n-1})dp_{n}(u_{n}|s_{n}-t_{n}|).
\end{eqnarray*}
Another  change of variables yields  (\ref{ineq:grr1}).
\end{proof}

\setcounter{equation}{0}
\section{Sample path joint H\"older continuity of random fields}\label{sec:random-fields}

\global\long\def\RR{\mathbb{R}}
\global\long\def\EE{\mathbb{E}}

In this section, given a continuous random field $W$, we study sample path joint
continuity property. The first application of Theorem \ref{lem:grrn}
is the following criteria for joint continuity of sample paths which
is similar to Kolmogorov continuity theorem, which we shall call
joint Kolmogorov continuity theorem.
\begin{thm}
\label{thm:Kolmogorov}Let $W$ be a continuous random field on $\RR^{n}$. Suppose
there exist positive constants $\alpha,\beta_{k}$ $(1\le k\le n)$
and $K$ such that for every $x,y$ in $[0,1]^{n}$,
\[
\EE\left[\left|\square_{y}^{n}W(x)\right|^{\alpha}\right]\le K\prod_{k=1}^{n}|x_{k}-y_{k}|^{1+\beta_{k}}.
\]
Then, for every $\epsilon=(\epsilon_{1},\dots,\epsilon_{n})$ with
$0<\epsilon_{k}\alpha<\beta_{k}$ $(1\le k\le n)$, there exist a
random variable $\eta$ with $\EE\eta^{\alpha}\le K$, such that the
following inequality holds almost surely
\[
|\square_{t}^{n}W(s)|\le C\eta(\omega)\prod_{k=1}^{n}|t_{k}-s_{k}|^{\beta_{k}\alpha^{-1}-\epsilon_{k}}
\]
for all $s,t$ in $[0,1]^{n}$, where $C$ is a constant defined by
\[
C=8^{n}4^{n/\alpha}\prod_{k=1}^{n}\left(1+\frac{2}{\beta_{k}-\alpha\epsilon_{k}}\right).
\]
\end{thm}
\begin{proof}
Let $\Psi(u)=|u|^{\alpha}$, $p_{k}(u)=|u|^{\gamma_{k}}$ where
$\gamma_{k}\in(\frac{2}{\alpha},\frac{2+\beta_{k}}{\alpha})$, $1\le
k\le n$. A direct application of  Theorem \ref{lem:grrn} gives
that   for all $s,t$ in $[0,1]^{n}$
\begin{eqnarray}
|\square_{s}^{n}W(t)| & \le &
8^{n}\prod_{k=1}^{n}\frac{\gamma_{k}|t_{k}-s_{k}|^{\gamma_{k}-\frac{2}
{\alpha}}}{\gamma_{k}-\frac{2}{\alpha}}\left(4^{n}\iint_{[0,1]^{2n}}
\frac{|\square_{y}^{n}W(x)|^{\alpha}}{\prod_{k=1}^{n}|x{}_{k}-y_{k}|^{\alpha\gamma_{k}}}dxdy
\right)^{\frac{1}{\alpha}}\,. \nonumber\\
\label{ineq:fst}
\end{eqnarray}
Let
\[
B(\omega)=\iint_{[0,1]^{2n}}\Psi\left(\frac{|\square_{y}^{n}W(x)|}{\prod_{k=1}^{n}p_{k}(x_{k}-y_{k})}\right)dxdy.
\]
From our assumption and Fubini-Tonelli's theorem,
\begin{eqnarray*}
\EE B & = & \iint_{[0,1]^{2n}}\frac{\EE|\square_{y}^{n}W(x)|^{\alpha}}{\prod_{k=1}^{n}|x_{k}-y_{k}|^{\alpha\gamma_{k}}}dxdy\\
 & \le & K\iint_{[0,1]^{2n}}\prod_{k=1}^{n}|x_{k}-y_{k}|^{1+\beta_{k}-\alpha\gamma_{k}}dxdy<\infty.
\end{eqnarray*}
Hence, the event $\Omega^{*}=\{\omega:B(\omega)<\infty\}$ has probability
one. Therefore for each $\omega$ in $\Omega^{*}$, the inequality
(\ref{ineq:fst}) gives
\[
|\square_{s}^{n}W(t,\omega)|\le8^{n}\prod_{k=1}^{n}\frac{\gamma_{k}|t_{k}-s_{k}|^{\gamma_{k}-\frac{2}{\alpha}}}{\gamma_{k}-\frac{2}{\alpha}}\left(4^{n}B(\omega)\right)^{\frac{1}{\alpha}}
\]
for every $s,t$ in $[0,1]^{n}$. For each $k$, the power
$\gamma_{k}-\frac{2}{\alpha}$ can be made arbitrarily close to
$\frac{\beta_{k}}{\alpha}$. This completes  the proof  with
$\eta=B^{1/\alpha}$.
\end{proof}

\begin{rem} The result obtained by Ral'chenko \cite{Ralchenko}
 was the inequality \eqref{ineq:fst}  in the case  $n=2$.
\end{rem}

\setcounter{equation}{0}

\section{Sample path joint continuity of Gaussian fields}\label{sec:gaussian-fields}
We now focus on sample path joint continuity of Gaussian random fields.
In case of Gaussian processes ($n=1$), one of the first sufficient
and necessary conditions for sample path continuity
was given by Fernique
\cite{fernique} (see also \cite{ferniquebook}).
Namely, let $p(u)$ be an increasing positive function
such that
\begin{equation}
\EE|W(x)-W(y)|^{2}\le p^{2}(|x-y|)\label{pdefn}
\end{equation}
for any pair $(x,y)$ in $[0,1]^{2}$. Then Fernique \cite{fernique}
showed that a sufficient condition for almost sure continuity of the process $(W(x), 0\le x\le 1)$ is
\[
\int_{0}^{1}\frac{p(u)}{u\sqrt{\log\frac{1}{u}}}du<\infty.
\]
In the original paper of
Garsia-Rodemich-Rumsey \cite{garsiarodemich}, the authors also observed
that the above condition is equivalent to the condition (by
integration by part)
\[
\int_{0}^{1}\sqrt{\log\frac{1}{u}}dp(u)<\infty.
\]
Later, it was shown that the above condition is also necessary
\cite{fernique66, marcus-shepp}. 
In case of Gaussian fields, recent progress on modulus of continuity
of Gaussian random fields has been reported in \cite{xiao,
ayache-xiao,MWX13}.

Let $W$ be a centered Gaussian random field with covariance function
\begin{equation}
\EE\left[W(x)W(y)\right]=Q(x,y).\label{cond:cov-1}
\end{equation}
We will always assume that $Q$ is a continuous function of $x$ and
$y$. For any fixed $x,y$, the random variable $\square_{y}^{n}W(x)$ is
also Gaussian with mean zero. In the following proposition, we compute its
variance.
\begin{prop}\label{prop:secondmoment}
Let $W$ be a centered Gaussian random field with covariance function
given by (\ref{cond:cov-1}). Then
\begin{equation}
\EE\left[|\square_{y}^{n}W(x)|^{2}\right]=\square_{(y,y)}^{2n}Q(x,x).\label{eq:secondmoment}
\end{equation}
Furthermore, if the covariance function $Q$ has  the
following product form
\begin{equation}
Q(x,y)=\prod_{k=1}^{n}Q_{k}(x_{k},y_{k})\, \label{cond:cov}
\end{equation}
then \eqref{eq:secondmoment} is simplified as
\begin{equation}
\EE\left[|\square_{y}^{n}W(x)|^{2}\right]=\prod_{k=1}^{n}\left[Q_{k}(x_{k},x_{k})-Q_{k}(x_{k},y_{k})-Q_{k}(y_{k},x_{k})+Q_{k}(y_{k},y_{k})\right].\label{eq:secondmoment2}
\end{equation}
\end{prop}
\begin{proof}
We calculate the variance directly as follows
\begin{eqnarray*}
\EE\left[|\square_{y}^{n}W(x)|^{2}\right] & = & \EE\left[\square_{y}^{n}W(x)\cdot\square_{y}^{n}W(x)\right]\\
 & = & \EE\left[\square_{(y,y)}^{2n}W(x)W(x)\right]\\
 & = & \square_{(y,y)}^{2n}\EE\left[W(x)W(x)\right]\\
 & = & \square_{(y,y)}^{2n}Q(x,x).
\end{eqnarray*}
The identity  \eref{eq:secondmoment} follows. To prove \eqref{eq:secondmoment2}, we notice that the pair of operators $(I-V_{k,(y,y)})(I-V_{n+k,(y,y)})$ transforms the $k$-th factor of $Q$ in \eqref{cond:cov} to
\[Q_{k}(x_{k},x_{k})-Q_{k}(x_{k},y_{k})-Q_{k}(y_{k},x_{k})+Q_{k}(y_{k},y_{k}).
\]Since the operators $I-V_{k,(y,y)}$, $(1\le k\le 2n)$ are commutative, we can write
\begin{eqnarray*}\square_{(y,y)}^{2n}Q(x,x)& = &\prod_{k=1}^n(I-V_{k,(y,y)})(I-V_{n+k,(y,y)})Q(x,x)\\
& = & \prod_{k=1}^{n}\left[Q_{k}(x_{k},x_{k})-Q_{k}(x_{k},y_{k})-Q_{k}(y_{k},x_{k})+Q_{k}(y_{k},y_{k})\right].
\end{eqnarray*}Hence, the identity \eqref{eq:secondmoment2} follows.
\end{proof}

\begin{defn}
Let $f$ be a continuous function on $\RR^{n}$. We call a set of
non-negative even functions $\{p_{1},\dots,p_{n}\}$ joint modulus of
continuity of $f$ if
\begin{description}
\item{(i)} For each $1\le k\le n$, $p_{k}(0)=0$, and $p_{k}$ is
non-decreasing and continuous.
\item{
(ii)}  For every pair $(s,t)$ in $\RR^{2n}$, the following
inequality holds
\[
|\square_{s}^{m}f(t)|\le\prod_{k=1}^{n}p_{k}(|t_{k}-s_{k}|)\,.
\]
\end{description}
\end{defn}
In view of Theorem  \ref{lem:grrn} and Theorem \ref{thm:Kolmogorov},
the continuity of sample paths is governed by the joint modulus of
continuity of $\square_{(y,y)}^{2n}Q(x,x)$. Such modulus of
continuity always exists. For instance, we can define a joint
modulus of continuity for $\square_{(y,y)}^{2n}Q(x,x)$ as follows.
We set
\[
p_{1}(u)=\sup_{x,y\in[0,1]^{n}:|x_{1}-y_{1}|\le u}\left[\square_{(y,y)}^{2n}Q(x,x)\right]^{\frac{1}{2}}.
\]
Given $p_{1},\dots,p_{k-1}$, define
\[
p_{k}(u)=\sup_{x,y\in[0,1]^{n}:|x_{k}-y_{k}|\le u}\frac{\left[\square_{(y,y)}^{2n}
Q(x,x)\right]^{\frac{1}{2}}}{\prod_{j=1}^{k-1}p_{j}(|x_{j}-y_{j}|)}\,,
\]
in which we have adopted the convention $0/0=0$. It follows immediately
that $p_{k}$'s are non-decreasing and continuous. Furthermore, we have
 $p_{k}(0)=0$
and
\begin{equation}
\square_{(y,y)}^{2n}Q(x,x)\le\prod_{k=1}^{n}p_{k}^{2}(|x_{k}-y_{k}|)\,.
\label{e.2.17}
\end{equation}
Namely,  $\{p_{1},p_{1},p_{2},p_{2},\dots,p_{n},p_{n}\}$ is a
modulus of continuity for $\square_{(y,y)}^{2n}Q(x,x)$.  We also
call $\{p_1, \cdots, p_n\}$ a modulus of continuity for
$\square_{(y,y)}^{2n}Q(x,x)$.

In the following theorem, we give a sufficient condition for almost
sure joint continuity of a Gaussian random field.
\begin{thm}
\label{thm:contW} Let $W$ be a continuous centered Gaussian random field with
covariance function given by (\ref{cond:cov-1}), and $p_{k}$ $(1\le
k\le n)$ be a modulus of continuity for
$\square_{(y,y)}^{2n}Q(x,x)$, namely the inequality
\eref{e.2.17} is satisfied. Suppose that
\begin{equation}\label{cond:p}
\sum_{k=1}^n \int_0^1\left(\log\frac1u\right)^{\frac12}dp_k(u)<\infty.
\end{equation}
Then, with probability one $W$ has joint continuous sample path. Furthermore, we have almost surely
that for any $\delta>0$,
\begin{equation}\label{eq:WLIL}
	 \sup_{0\le |x-y|\le \delta }\frac{|\square^n_yW(x)|}{h(x,y)}\le c_{n, \delta} \,,
\end{equation}
where $h(x,y)$ is the function
\begin{equation}
	h(x,y)={\prod_{k=1}^np_k(|x_k-y_k|)}\sqrt{\log \prod_{j=1}^n{\frac1{|x_j-y_j|}}}\,,
\end{equation}
$c_{n, \delta}$ is a random variable, depending on $n$ and $\de$, and $\EE e^{c_{n,\delta}^2}<\infty $ . Moreover, there exists a constant $\kappa_n$ such that
\begin{equation}\label{ineq.WLIL2}
	\lim_{\delta\downarrow 0}\sup_{|x-y|\le \delta}\frac {|\square_y^n W(x)|}{h(x,y)}\le \kappa_n
\end{equation}
almost surely.

\end{thm}
\begin{proof}
We set $\Psi(x)=e^{x^{2}/4}$ and
\[
B(\omega)=\iint_{[0,1]^{2n}}\exp\left[\frac{|\square_{y}^{n}W(x)|^{2}}{4\prod_{k=1}^{n}p_{k}^{2}(|x_{k}-y_{k}|)}\right]dxdy.
\]
Theorem \ref{lem:grrn}   gives
\begin{eqnarray}\label{eq:Wcont}
	|\square_{y}^{n}W(x)|&\le&2\cdot8^{n}\int_{0}^{|x_{1}-y_{1}|}\cdots
	\int_{0}^{|x_{n}-y_{n}|}\left(\log\frac{1}{u_{1}^{2}\cdots u_{n}^{2}}\right)
	^{\frac{1}{2}}dp_{1}(u_{1})\cdots dp_{n}(u_{n})\nonumber\\
	&&\quad +\sqrt {\log(4^{n}B(\omega))}\prod_{k=1}^n p_k(|x_k-y_k|)
\end{eqnarray}
for $\omega$ such that $B(\omega)$ is finite.

It is elementary to see that
\begin{equation*}	
		\lim_{|x-y|\to0}\frac1{h(x,y)} \int_{0}^{|x_{1}-y_{1}|}\cdots \int_{0}^{|x_{n}-y_{n}|}\left(\log\frac{1}{u_{1}^{2}\cdots u_{n}^{2}}\right)
	^{\frac{1}{2}}dp_{1}(u_{1})\cdots dp_{n}(u_{n})=c_n
\end{equation*}	
for some constant $c_n$ and
\begin{equation*}
	\lim_{|x-y|\to0}\frac{\prod_{k=1}^n p_k(|x_k-y_k|)}{h(x,y)}=0\,.
\end{equation*}
From \eqref{eq:Wcont} and the two facts above, the estimates \eqref{eq:WLIL} and \eqref{ineq.WLIL2} follow easily.

To see \eqref{eq:Wcont} indeed
holds for almost every $\omega$ (and hence \eqref{eq:WLIL} and \eqref{ineq.WLIL2}), it is sufficient to show that $B$
has finite expectation. We notice that the random variable
\[
N=\frac{\square_{y}^{n}W(x)}{\prod_{k=1}^{n}p_{k}(|x_{k}-y_{k}|)}
\]
is Gaussian, has mean zero and variance less than or equal to one.
Thus, an application of Stirling's formula gives
\begin{eqnarray*}
\EE\exp\left(\frac{N^{2}}{4}\right) & = & \sum_{k=0}^{\infty}\frac{\EE N^{2k}}{4^{k}k!}\\
 & = & 1+\sum_{k=1}^{\infty}\frac{(2k)!}{8^{k}\left(k!\right)^{2}}(\EE N^{2})^k\\
 & \le & 1+\frac{1}{2}\sum_{k=1}^{\infty}8^{-k}=\frac{15}{14}.
\end{eqnarray*}
Hence
\[
\EE B=\iint_{[0,1]^{2n}}\EE\exp\left(\frac{N^{2}}{4}\right)dxdy\le\frac{15}{14}
\]
and the proof is complete.
\end{proof}
\begin{rem}\label{rem:1} Suppose that $W(x)=0$ whenever $x$ has at least one zero coordinate. Let $\sigma(x,y)$ be the function defined below
	\begin{equation}\label{e.sigma} 
		\sigma(x,y)=\sum_{k=1}^n \left(\prod_{j\neq k}p_j(|z_{j,k}|)\right)\left|\log \prod_{j\neq k}|z_{j,k}|\right|^{1/2}p_k(|x_k-y_k|) |\log|x_k-y_k||^{1/2}\,,
	\end{equation}
	where $z_{j,k}=x_j$ if $j<k$ and $z_{j,k}=y_j$ if $j>k$.
	The inequality \eqref{ineq.WLIL2} implies the following estimate which usually appears in literature
	\begin{equation}\label{eq:WLIL2}
		\lim_{\delta\downarrow0}\sup_{\substack{|x|\le1,|y|\le1\\|x-y|\le \delta}}\frac {|W(x)-W(y)|}{\sigma(x,y)}\le \kappa_n\,.
	\end{equation}
	Indeed, fix $\omega$ such that \eqref{ineq.WLIL2} holds and $\delta$ sufficiently small, for every $x,y$ in $[0,\delta]^n$, with $x$ and $(0,0,\dots,0,y_n)$, the estimate \eqref{eq:WLIL} gives the following estimate for the increment along an edge of the $n$-dimensional rectangle $[x_1,y_1]\times\cdots\times[x_n,y_n]$
	\begin{multline*}\left|W(x_1,\cdots,x_n)-W(x_1,\cdots,x_{n-1},y_n)\right|\\
		\le c_{n,\delta}\left(\prod_{k=1}^{n-1}p_k(|x_k|)\right)\left|\log{\prod_{k=1}^{n-1}|x_k|}\right|^{1/2}p_n(|x_n-y_n|)\left|\log
		|x_n-y_n|\right|^{1/2}\,.
	\end{multline*}
	Similarly, we can obtain analogue estimates along any edge of the $n$-dimensional rectangle $[x_1,y_1]\times\cdots\times[x_n,y_n]$. The increment along the diagonal is majorized by the total increments along all the edges connecting $x$ and $y$. Hence, this argument yields the following estimate
	\begin{equation}\left|W(x)-W(y)\right|\le c_{n,\delta} \sigma(x,y)
	\end{equation}
	which implies \eqref{eq:WLIL2}.

\end{rem}


As an application of the above theorem, we  obtain joint continuity
for sample paths of fractional Brownian field, as mentioned in
\eref{e.1.5}. 
\begin{cor}\label{cor:fbm}
Let $W^{H}$ be a fractional Brownian field on $\RR^{n}$ with Hurst
parameter $H=(H_{1},\dots,H_{n})$. Then, for any $\de>0$ the following inequality holds almost surely
\begin{equation}\label{e.4.7}
	 \sup_{|x-y|\le\delta}\frac{|\square_y^nW^H(x)|}{h^H(x,y)}\le c_{n,\de}
\end{equation}
where $h^H(x,y)$ is the function
\begin{equation}
	h^H(x,y)=\prod_{k=1}^n |x_k-y_k|^{H_k}\left|\log \prod_{j=1}^n|x_j-y_j|\right|^{1/2}
\end{equation}
for some finite random variable $c_{n,\de} $ depending on $n$ and $\de$ such that $\EE e^{c_{n,\delta}^2}<\infty$. Moreover, there is a constant $\kappa_n$ such that
 \begin{equation}
 	\lim_{\delta\downarrow0}\sup_{|x-y|\le \delta}\frac{|\square_y^nW^H(x)|}{h^H(x,y)}\le \kappa_n
 \end{equation}
 almost surely.
\end{cor}
\begin{proof}
The covariance function of a fractional Brownian field is given by
\[
\EE\left[W^{H}(x)W^{H}(y)\right]=\prod_{k=1}^{n}R_{k}(x_{k},y_{k}),
\]
where
\[
R_{k}(s,t)=\frac{1}{2}\left[|s|^{2H_{k}}+|t|^{2H_{k}}-|s-t|^{2H_{k}}\right]\,,
\quad \forall \ s, t\in \RR\,.
\]
By Proposition \ref{prop:secondmoment}, we obtain the second moment
for $\square_{y}^{n}W^{H}(x)$
\[
\EE|\square_{y}^{n}W^{H}(x)|^{2}=\prod_{k=1}^{n}|x_{k}-y_{k}|^{2H_{k}}.
\]
This means that $p_i(u)=u^{H_i}\,, i=1, 2, \cdots\,, n$ are the
modulus of continuity of $\square_{(y,y)}^{2n}Q(x,x)$. Now the corollary is a direct consequence of Theorem
\ref{thm:contW}.
\end{proof}
\begin{rem}
\begin{enumerate}
	\item In case of $n$-parameter Wiener process, that is when $H_1=\cdots=H_n=1/2$, the above corollary is comparable to a result of S. Orey and W. E. Pruitt in \cite[Theorem 2.1]{Orey-Pruitt}.
	\item As in Remark \ref{rem:1}, let $\sigma^H(x,y)$ be the function
	\begin{equation}\label{sigmaH}
		\sigma^H(x,y)=\sum_{k=1}^n \left(\prod_{j\neq k}|z_{j,k}|^{H_j}\right)\left|\log \prod_{j\neq k}|z_{j,k}|\right|^{\frac12}|x_k-y_k|^{H_k} |\log|x_k-y_k||^{\frac12}\,,
	\end{equation} 
	where $z_{j,k}=x_j$ if $j<k$ and $z_{j,k}=y_j$ if $j>k$.
	Then the previous result implies the following
	\begin{equation}
		 \lim_{\delta\downarrow0} \sup_{\substack{|x|\le1,|y|\le1\\|x-y|\le \delta}}\frac{|W^H(x)-W^H(y)|}{\sigma^H(x,y)}\le c_{n}
	\end{equation}
	where $c_n$ is some constant.
\end{enumerate}
\end{rem}

\setcounter{equation}{0}
\section{Stochastic heat equations with additive space time white noise}\label{sec:spde}

In this section let us consider the following one dimensional
stochastic differential equation
\begin{equation}
\left\{
\begin{array}{ll}
\frac{\partial u}{\partial t}=\frac12\Delta u +  \dot W& \qquad 0<t\le T\,, \ y\in \RR\\ \\
u(0, y)=0 &\qquad y\in \RR\,,
\end{array}\right.
\end{equation}
where $\Delta u=\frac{\partial ^2}{\partial y^2} u$, $W$ is space
time   standard Brownian sheet,  and $\dot W=\frac{\partial
^2}{\partial t\partial y} W$. Let $\displaystyle
p_t(y)=\frac{1}{\sqrt{2\pi t}}e^{-\frac{y^2}{2t}}$.  Then the (mild)
solution of the above equation is given by
\[
u(t, y)=\int_0^t \int_\RR p_{t-r} (y-z) W(dr, dz)\,,
\]
where the above integral is the usual (It\^o) stochastic integral
(however, the integrand is simple. It is a deterministic function).
The solution $u(t, y)$  is a Gaussian random field. It is known that
$u(t, y)$ is H\"older continuous of exponent $\frac14-$ for time
parameter and $\frac12-$ for space parameter.  Namely, for any
$\al<1/4$ and any $\be<1/2$,  there is a random constant $C_{\al,
\be}$ such that
\begin{equation}\label{ineq:uholder}
|u(t, y)-u(s, x)|\le C_{\al, \be}
\left(|t-s|^{\al}+|x-y|^{\be}\right)\,.
\end{equation}

We are interested in the  joint H\"older continuity of the solution
$u(t,y)$. We need the   following simple technical lemma.
  \begin{lem}
\label{lem:iabexp}Let $a, b,\delta$ be some positive numbers,  where $a<b$,  and let
$I,J$ be the integrations
\[
I=\int_{a}^{b}\frac{1}{\sqrt{r}}\left(1-e^{-\frac{\delta^2}{2r}}\right)dr,
\]
\[J=\int_{0}^{a}\frac{1}{\sqrt{r}}\left(1-e^{-\frac{\delta^2}{2r}}\right)dr.
\]
Then for every $\alpha\in[0,1/2]$,
\[
2(\sqrt{b}-\sqrt{a})\left(1-e^{-\frac{\delta^2}{2b}}\right)\le I\le2(\sqrt{b}-\sqrt{a})\left(1-e^{-\frac{\delta^2}{2a}}\right)
\]
and
\[J\le c_\alpha \delta^{2\alpha}a^{1/2-\alpha}.
\]
\end{lem}
\begin{proof}  On the interval $a\le r\le b$, we have
$\displaystyle 1-e^{-\frac{\delta^2}{2b}}\le 1-e^{-\frac{\delta^2}{2r}}\le 1-e^{-\frac{\delta^2}{2a}}$.  The estimate for $I$ is then a straightforward  consequence. To estimate $J$, we first use integration by part to obtain
\begin{eqnarray*}
J& = &2\sqrt{r}\left(1-\ee r\right)_{r=0}^{r=a}+\int_{0}^{a}\left(\frac{d}{dr}\ee r\right)2\sqrt{r}dr\\
 & = & 2\sqrt{a}\left(1-\ee a\right)+\delta^2\int_{0}^{a}\ee r r^{-3/2}dr.
\end{eqnarray*}
By a change of variable $x=\frac{\delta}{\sqrt{2r}}$, we see that
\[J=2\sqrt{a}\left(1-\ee a\right)+2\sqrt{2}\delta\int_{\frac\delta{\sqrt{2a}}}^\infty e^{-x^2} dx.
\]
If $\frac\delta{\sqrt{2a}}\ge1$, since $\lim_{t\to\infty}\frac{\int_{t}^\infty e^{-x^2} dx}{(1-e^{-t^2})t^{-1}}=0$, $J$ is majorized by
\[J\le c\sqrt{a}\left(1-\ee a\right).\]
If $\frac\delta{\sqrt{2a}}\le1$, the integration $\int_{\frac\delta{\sqrt{2a}}}^\infty e^{-x^2} dx$ is bounded by $\sqrt\pi/2$, thus $J$ is majorized by
\[J\le 2\sqrt{a}\left(1-\ee a\right)+c\delta.\]Therefore, for any $0\le \alpha\le1/2$, employing the elementary inequality $1-e^{-x}\le c_\alpha x^\alpha$, we obtain
\[J\le c_\alpha\delta^{2\alpha}a^{1/2-\alpha}\]and the lemma follows.
\end{proof}

\begin{theorem}\label{thm:spde} For every $\alpha$ in $[0,1/4]$ and $\de>0$, there is a finite random variable $c_{\alpha,\de} $ depending on $\alpha$ and $\de$ such that the following estimate holds almost surely
\begin{equation}\label{ineq:ustxy} \sup_{|t-s|\le\delta,|x-y|\le\delta}\frac{|u(t,y)-u(t,x)-u(s,y)+u(s,x)|}{|t-s|^{\frac14-\alpha} |x-y|^{2\alpha}
{\left|\log\left( |t-s||x-y|\right)  \right|^{\frac12}  }}\le c_{\alpha,\de}\,.
\end{equation}
Moreover, for some absolute constant $\kappa_{\alpha}$, the following inequality holds
\begin{equation}
	 \lim_{\delta\downarrow 0}\sup_{|t-s|\le\delta,|x-y|\le\delta}\frac{|u(t,y)-u(t,x)-u(s,y)+u(s,x)|}{|t-s|^{\frac14-\alpha} |x-y|^{2\alpha}
{\left|\log\left( |t-s||x-y|\right)  \right|^{\frac12}  }}\le \kappa_{\alpha}\,.
\end{equation}
\end{theorem}
\begin{proof} $u(t, y)$ is a mean zero Gaussian field.
The covariance of $u(t,y)$ and $u(s, x)$ is given by
\begin{eqnarray*}
\EE[u(s,x)u(t,y)] & = & \int_{\RR^{2}}\chi_{[0,s]}(r)\chi_{[0,t]}(r)p_{s-r}(x-z)p_{t-r}(y-z)drdz\\
 & = & \int_{\RR^{2}}f(s,x)f(t,y)drdz\,,
\end{eqnarray*}
where $f(s,x)=\chi_{[0,s]}(r)p_{s-r}(x-z)$.

We calculate the second moment of $\square_{(s,x)}^{2}u(t,y)$ as
follows
\begin{eqnarray*}
\EE\left[\square_{(s,x)}^{2}u(t,y)\right]^{2} & = & \EE\square_{(s,x)}^{2}u(t,y)\square_{(s,x)}^{2}u(t,y)\\
 & = & \EE\square_{(s,x,s,x)}^{4}u(t,y)u(t,y)\\
 & = & \square_{(s,x,s,x)}^{4}\EE\left[u(t,y)u(t,y)\right]\\
 & = & \square_{(s,x,s,x)}^{4}\int_{\RR^{2}}f(t,y)^{2}drdz\\
 & = & \int_{\RR^{2}}\square_{(s,x,s,x)}^{4}\left[f(t,y)f(t,y)\right]drdz\\
 & = & \int_{\RR^{2}}\left[\square_{(s,x)}^{2}f(t,y)\right]\left[\square_{(s,x)}^{2}f(t,y)\right]drdz\\
 & = & \int_{\RR^{2}}\left[\square_{(s,x)}^{2}f(t,y)\right]^{2}drdz\,,
\end{eqnarray*}
where
\begin{eqnarray*}
\left[\square_{(s,x)}^{2}f(t,y)\right]^{2} & = & \left[f(s,x)-f(t,x)-f(s,y)+f(t,y)\right]^{2}\\
 & = & f(s,x)^{2}+f(t,x)^{2}+f(s,y)^{2}+f(t,y)^{2}\\
 &  & -2f(s,x)f(t,x)-2f(s,x)f(s,y)+2f(s,x)f(t,y)\\
 &  & +2f(t,x)f(s,y)-2f(t,x)f(t,y)-2f(s,y)f(t,y).
\end{eqnarray*}
Taking the integration with respect to $z$ and  using the following identity
\[
\int_{\RR}p_{a}(z-x)p_{b}(z-y)dz=p_{a+b}(x-y)
\]
we obtain
\begin{eqnarray*}
& &\EE\left[\square_{(s,x)}^{2}u(t,y)\right]^{2}\\ & = & \int_{\RR}\left[2\chi_{[0,s]}(r)p_{2s-2r}(0)+2\chi_{[0,t]}(r)p_{2t-2r}(0)\right]dr\\
 &  & +\int_{\RR}\left[-2\chi_{[0,s\wedge t]}p_{s+t-2r}(0)-2\chi_{[0,s]}p_{2s-2r}(x-y)+2\chi_{[0,s\wedge t]}p_{s+t-2r}(x-y)\right]dr\\
 &  & +\int_{\RR}\left[2\chi_{[0,s\wedge t]}p_{s+t-2r}(x-y)-2\chi_{[0,t]}p_{2t-2r}(x-y)-2\chi_{[0,s\wedge t]}p_{s+t-2r}(0)\right]dr\\
 & = & 2\int_{0}^{s}\left[p_{2s-2r}(0)-p_{2s-2r}(x-y)\right]dr+2\int_{0}^{t}\left[p_{2t-2r}(0)-p_{2t-2r}(x-y)\right]dr\\
 &  & -4\int_{0}^{s\wedge t}\left[p_{s+t-2r}(0)-p_{s+t-2r}(x-y)\right]dr.
\end{eqnarray*}
By change of variables $u=2s-2r$, $v=2t-2r$ and $w=s+t-2r$ in the
above corresponding integrals respectively and noticing  that $s+t-2(s\wedge t)=|t-s|$,
we get
\begin{eqnarray*}
\EE\left[\square_{(s,x)}^{2}u(t,y)\right]^{2} & = & \int_{0}^{2s}\left[p_{u}(0)-p_{u}(x-y)\right]du+\int_{0}^{2t}\left[p_{v}(0)-p_{v}(x-y)\right]dv\\
 &  & -2\int_{|s-t|}^{s+t}\left[p_{w}(0)-p_{w}(x-y)\right]dw\\
 & = & \left(\int_{s+t}^{2(s\vee t)}-\int_{2(s\wedge t)}^{s+t}+2\int_{0}^{|s-t|}\right)\left[p_{r}(0)-p_{r}(x-y)\right]dr.
\end{eqnarray*}
By Lemma \ref{lem:iabexp}, we see that
\begin{multline*}
\left(\int_{s+t}^{2(s\vee t)}-\int_{2(s\wedge t)}^{s+t}\right)\left[p_{r}(0)-p_{r}(x-y)\right]dr\\
\le\frac{1}{\sqrt{2\pi}}\left(1-e^{-\frac{(x-y)^{2}}{2(s+t)}}\right)\left(\sqrt{2s}+\sqrt{2t}-2\sqrt{s+t}\right)\le0.
\end{multline*}
and
\[\int_{0}^{|s-t|}\left[p_{r}(0)-p_{r}(x-y)\right]dr\le c_\alpha|x-y|^{2\alpha}|s-t|^{1/2-\alpha}
\]
for every $\alpha$ in $[0,1/2]$.
Thus
\begin{eqnarray}
\EE\left[\square_{(s,x)}^{2}u(t,y)\right]^{2}
&\le&2 \int_{0}^{|s-t|}\left[p_{r}(0)-p_{r}(x-y)\right]dr\label{e.touse}\\
&\le&  c_\alpha|x-y|^{2\alpha}|s-t|^{1/2-\alpha}.\nonumber
\end{eqnarray}
An application of Theorem \ref{thm:contW} immediately gives the desired result.
\end{proof}

\begin{rem}
Using the method in Remark \ref{rem:1}, the above result implies there is a constant $c$ such that
\begin{equation}\label{ineq:uLIL2}
	\lim_{\delta\downarrow0}\sup_{\substack{|t|\le1,|s|\le1\\|x|\le1,|y|\le1\\|t-s|\le \delta,|x-y|\le \delta}}\frac {|u(s,x)-u(t,y)|}{
	|s-t|^{\frac14}\sqrt{\log\frac{1}{|x||s-t|}}+|x-y|^{\frac12}
\sqrt{\log\frac{1}{|x-y||t|}}}\le c
\end{equation}
which is sharper than \eref{ineq:uholder}. 
\end{rem}
\begin{rem} After the completion of this paper  it is communicated to us that recently,   M. Meerschaert, W. Wang and Y. Xiao obtained  (see  \cite{MWX13},
Theorem 4.1 and see also Theorem 6.1 for fractional multiparameter Brownian motion)
the following result.   Let  $W$ be an  $n$-parameter  Gaussian process with
mean zero  and
\[
\rho^2(x,y)=\EE\left[ |W(x)-W(y)|^2\right]\,.
\]
Assume there are positive constants  $H_1,
\cdots\,,  H_n\in (0,1]$ and  positive constants $C_1<C_2$ such that
\begin{equation}\label{aniso}
	C_1\left(\sum_{j=1}^n |x_j-y_j|^{H_j}\right)^2\le \rho^2(x,y)\le C_2\left(\sum_{j=1}^n |x_j-y_j|^{H_j}\right)^2\,.	
\end{equation}
Assume further that $W$ satisfies some conditions  that we don't repeat here and refer interested readers to \cite{MWX13}. Let $I=[a,1]^n$ where $a\in(0,1)$ is a constant. Then
	\begin{equation}
	\lim_{\delta\downarrow0}\sup_{\substack{x,y\in I\\|x-y|\le\delta}}\frac{|W (x)-W (y)|}{\be (x,y)}=\kappa    \label{mwx}
	\end{equation}
	for some positive   constant $\kappa$,  where
\[
\be(x,y)=\rho(x,y)\sqrt{\log (1+\rho(x,y)^{-1})} \,.
\]
It is obvious that as $|x-y|\rightarrow 0$, we have
\[
\be(x,y)\approx \rho(x,y)\sqrt{|\log (|x-y|)|} \,.
\]
Moreover, given \eqref{aniso}, $\beta(x,y)$ has the same order as $\sigma^H(x,y)$ in \eqref{sigmaH} when $x,y$ are bounded and $|x-y|\to0$.
Thus the  identity  \eref{mwx} says that our inequality    \eqref{eq:WLIL2} is sharp. Besides, \eqref{eq:WLIL2} does not require $x,y$ to be bounded away from $0$.
We  conjecture that the inequality \eqref{ineq:ustxy}   is also sharp.  Moreover,
it is interesting to know if an analogous identity to \eref{mwx} holds for the increments over the rectangles
of type  \eqref{ineq:ustxy}  or not. Namely, for any $\al\in [0, 1/4]$,  is  there  a positive constant  $\kappa_\al $ such that
\begin{equation}\label{conjecture}
\lim_{\de\downarrow 0} \sup_{|t-s|\le\delta,|x-y|\le\delta}\frac{|u(t,y)-u(t,x)-u(s,y)+u(s,x)|}{|t-s|^{\frac14-\alpha} |x-y|^{2\alpha}
{\left|\log\left( |t-s||x-y|\right)  \right|^{\frac12}  }}= \kappa_{\alpha }\,?
\end{equation}
As it is well-known the Garsia-Rodemich-Rumsey inequality gives only
the upper bound. It has not been powerful to obtain  the lower bound.
Therefore,  one has to  attack the above problem \eref{conjecture}  using other means. As a confirmative example, we remark that in the case of Brownian sheet on $\RR^2$, G. J. Zimmerman showed in \cite{Zimmerman} that
\begin{equation}
	\limsup_{\substack{|x_1-y_1|=\delta_1\downarrow 0\\|x_2-y_2|=\delta_2\downarrow0} }\frac{|\square^2_yW(x)|}{[2 \delta_1 \delta_2\log(1/(\delta_1 \delta_2))]^\frac12}=1\,.
\end{equation}
\end{rem}

\begin{rem} If one prefers to write one inequality rather than arbitrary $\al$ in
Theorem \ref{thm:spde}, one can write the
  inequality \eref{e.touse}   as
\begin{eqnarray*}
\EE\left[\square_{(s,x)}^{2}u(t,y)\right]^{2}
&\le&2 \int_{0}^{|s-t|} \frac{1}{\sqrt {2\pi r}} \left[1-e^{-\frac{|y-x|^2}{2r}}\right]dr \\
&\le&\sqrt {\frac2\pi} |x-y|\int_0^{\frac{|t-s|}{|x-y|^2}} r^{-1/2} (1-e^{-\frac1r}) dr\\
&=& |x-y| \rho\left(\frac{|t-s|}{|x-y|^2}\right)\,,
\end{eqnarray*}
where $\displaystyle
\rho(u)=\sqrt {\frac2\pi}  \int_0^u r^{-1/2} (1-e^{-\frac1r}) dr$ which is of the order $\sqrt u$ as $u\rightarrow 0$ and bounded as $u\rightarrow \infty$.  It is obvious that
$|x-y| \rho\left(\frac{|t-s|}{|x-y|^2}\right)$ goes to $0$ if one of $|t-s|$ and
$|x-y|$ goes to $0$.
\end{rem}

\noindent{\bf Acknowledgment}: The authors thank the referee for careful reading of the paper
and for some constructive comments.



\begin{thebibliography}{9}
\bibitem{ayache-antoine-leger}  Ayache, A.; Leger, S. and  Pontier, M.  Drap brownien fractionnaire.  Potential Anal. 17 (2002), no. 1, 31-43.
\bibitem{ayache-xiao} Ayache, A. and Xiao, Y.
Asymptotic properties and Hausdorff dimensions of fractional Brownian sheets.  J. Fourier Anal. Appl. 11 (2005), no. 4, 407-439.
\bibitem{fernique} Fernique, X.
Continuit\'e des processus Gaussiens.
C. R. Acad. Sci. Paris 258 (1964),  6058-6060.
\bibitem{fernique66}  Fernique, X.  S\'eries de distributions
al\'eatoires ind\'ependantes.  C. R. Acad. Sci. Paris S\'er. A-B 263 (1966),  A674-A677.
\bibitem{ferniquebook}  Fernique, X.  Fonctions al\'eatoires gaussiennes,
vecteurs al\'eatoires gaussiens.  Universit\'e  de Montr\'eal,
Centre de Recherches Math\'ematiques, Montreal,  QC,  1997.

\bibitem{garsiarodemich} Garsia, A. M.; Rodemich, E.;  and Rumsey, H., Jr.
A real variable lemma and the continuity of paths of some Gaussian processes.
Indiana Univ. Math. J.  20 (1970/1971),  565-578.
\bibitem{MWX13} Meerschaert, M.  M.; Wang, W.; Xiao, Y.
Fernique-type inequalities and moduli of continuity for anisotropic Gaussian random fields.
Trans. Amer. Math. Soc. 365 (2013), no. 2, 1081-1107.
\bibitem{marcus-shepp}  Marcus, M. B.; Shepp, L. A.
Continuity of Gaussian processes. Trans. Amer. Math. Soc. 151 (1970),   377-391.
\bibitem{Orey-Pruitt} Orey, Steven; Pruitt, William E.
Sample functions of the {N}-parameter Wiener process. 
Ann. Probability 1 (1973), no. 1, 138-163. 
\bibitem{Ralchenko} Ral'chenko, K. V. The two-parameter Garsia-Rodemich-Rumsey inequality and its application to fractional Brownian fields.  Theory Probab. Math. Statist. No. 75 (2007), 167-178.
 \bibitem{xiao} Xiao, Y.
Sample path properties of anisotropic Gaussian random fields. A minicourse on stochastic partial differential equations, 145-212,
Lecture Notes in Math., 1962, Springer, Berlin, 2009.
\bibitem{Zimmerman}Zimmerman, Grenith J.
Some sample function properties of the two-parameter Gaussian process.
Ann. Math. Statist. 43 (1972), 1235-1246.




\end{thebibliography}
\end{document}